\documentclass{article}
  \usepackage[reqno, namelimits, sumlimits]{amsmath}
        \usepackage{amssymb, amsfonts}
        \usepackage{mathrsfs}
        \usepackage[margin=1.5in, paperwidth=8.5in, paperheight=11.0in]{geometry}

\newcommand{\be}{\begin{equation}}
\newcommand{\ee}{\end{equation}}
\newcommand{\bb}{\bigskip}
\newcommand{\la}{\label}

\renewcommand{\div}{\,{\rm div}\,}
\newcommand{\curl}{\,{\rm curl}\,}

\newcommand{\rf}[1]{(\ref{#1})}

\newcommand{\ve}{\varepsilon}

\newcommand{\ZZ}{\bf Z}

\newcommand{\R}{{\mathbf R}}

\newcommand{\PP}{\mathcal P}

\newcommand{\EE}{\mathcal E}

\newcommand{\dd}{\Delta}
\newcommand{\om}{\omega}
\newcommand{\sm}{\smallskip}

\newcommand{\tor}{{\bf T}^2}
\newcommand{\QED}{\mbox{}\hfill$\Box$}
\newcommand{\ue}{u_\ve}
\newcommand{\inttor}{\int_{\tor}}
\newcommand{\ome}{\om_\ve}
\newcommand{\zt}{{\ZZ}^2}
\newcommand{\tjd}{(t_1,t_2)}
\newcommand{\sconv}{S^{\rm conv}}
\newcommand{\omk}{\om_K}

\usepackage{amsmath}
\numberwithin{equation}{section}
\newtheorem{theorem}{Theorem}[section]

\newtheorem{corollary}{Corollary}[section]
\newtheorem{proposition}{Proposition}[section]

\newtheorem{remark}{Remark}[section]
\newenvironment{proof}
{\begin{trivlist}\item[]{\bf Proof.}}
{\hspace*{\fill}\QED\end{trivlist}}
\usepackage{authblk}

\begin{document}
\title{On 2d incompressible Euler equations with partial damping}

%\author[1]{ Tarek Elgindi}
%\author[2]{ Wenqing Hu}
%\author[2]{ Vladimir Sverak}
%\affil[1]{{\small Princeton University}}
%\affil[2]{{\small University of Minnesota}}

\author{Tarek~Elgindi,\,\, Wenqing~Hu,\,\,
 Vladim\'i{}r~\v Sver\'ak}
 \date{}
%\footnote{The research was  supported in part by The research was was supported in part by grants DMS 1362467 and DMS 1159376 from the National Science Foundation}

\maketitle

\begin{abstract}
We consider various questions about the 2d incompressible Navier-Stokes and Euler equations on
a~torus when dissipation is removed from or added to some of the Fourier modes.
\end{abstract}

\section{Introduction}\label{intro}

Our motivation for this paper comes from the theory of turbulence for the 2d Navier-Stokes equation
\be\la{1}
\begin{array}{rcl}
u_t+u\nabla u +\nabla p-\nu \dd u & = & f\,,\\
\div u & = & 0\,,\\
u|_{t=0} & = & u_0
\end{array}
\ee
on the torus
$
\tor=\R^2/2\pi \ZZ^2\,.
$
  We have mostly in mind the situation when spatial Fourier modes of $f$ and $u_0$ are assumed to be supported only on relatively low frequencies.
One can consider both deterministic and ``white noise in time" $f$. The canonical (conjectural) picture due to Kraichnan~\cite{[Kraichnan 1967]} (perhaps under some genericity assumptions) is that of the energy and enstrophy cascades which spread the excitations to other Fourier modes through the nonlinearity.

What happens when we remove the dissipation on a certain set of $K\subset \ZZ^2$  frequencies? Using standard notation for Fourier series,
\be\la{2}
u(x,t)=\frac 1{(2\pi)^2} \sum_{k\in{\ZZ}^2} \hat u(k,t) e^{ikx}\,,
\ee
we define an operator $Y$ by
\be\la{3}
\widehat{(Y\!u)}(k)=
\begin{cases}
-|k|^2 \,\hat u(k)  \, & k\notin K\,,\\
\hskip10pt{0}   & k\in K\,.
\end{cases}
\ee
and consider the following modification of~\rf{1}
\be\la{4}
\begin{array}{rcl}
u_t+u\nabla u +\nabla p-\nu Y\! u & = & f\,,\\
\div u & = & 0\,,\\
u|_{t=0} & = & u_0\,.
\end{array}
\ee
We will assume that $K$ is symmetric (i.\ e.\ invariant under $k\to -k$), to keep the solution real-valued.
One of the simplest cases should be when we have $K=\{\bar k,-\bar k\}$ for some high freguency $\bar k$. When $\bar k$ is quite higher than the non-zero (spatial) frequencies of the forcing $f$, the usual heuristics used in turbulence investigations suggests that the effect of removing the dissipation from the frequencies in $K$ should not be  dramatic.  The non-linearity will still spread the energy and enstrophy between many frequencies, and there should be enough dissipation in the system to keep the solution bounded for all time. The situation is somewhat similar to adding partial damping to a system for which the  principles of Statistical Mechanics work: the interactions in such a system should tend to distribute energy uniformly between all degrees of freedom, and hence a system for which some of the degrees of freedom are forced while some other degrees of freedom are damped  should still reach some kind of dynamical equilibrium (at least if there is enough interaction between the damped and forced parts of the system).
 While this sounds very plausible, establishing such statements rigorously is usually difficult.

In the context of the 2d turbulence discussed above, there might be interesting issues related to Kraichnan's downward cascade of energy, see~\cite{[Kraichnan 1967]}, and if in the case $K=\{\bar k, -\bar k\}$  we take $\bar k$ to be one of the lowest non-trivial frequencies, the solution  will presumably not stay bounded generically, even if the forcing acts ``far away" in the Fourier space.

Ultimately we would like to study the situation for various sets of frequencies $K$ and various forces $f$ (both deterministic and random) and some of the problems seem to be approachable by existing methods. In particular, it might be interesting to investigate under which assumptions one can extend the results of Hairer-Mattingly~\cite{HM} on invariant measures for the stochastic forcing, or the results of Constantin-Foias~\cite{Constantin} on attractors for deterministic forcing, see also Ladyzhenskaya's work~\cite{Ladyzhenskaya2}. Another set of results which might be interesting to consider in the context here are those of Kuksin~\cite{Kuksin}.

 In this paper our goals are much more modest and we will deal only with the  simpler but still interesting situation when $f=0$, leaving the case  $f\ne0$ to future work. It turns out that even in the case of no forcing there are still some interesting open problems. For example, one can ask under which assumptions on $K$ the ``generic" solutions\footnote{
 Even for the simplest non-empty $K$ which are invariant under in involution $k\to-k$ one has non-trivial steady states (e.\ g.\ vector fields with stream functions $A\cos (kx)+B\sin (kx)$), and hence the it seems reasonable to consider generic solutions in the above question.
} of~\rf{5} approach zero (and in which norms) as $t\to\infty$.
This seems to be a difficult question to which we do not know the answer. It is clearly related to the problem of energy or enstrophy transfer between various Fourier modes, which is also one of the important themes of turbulence investigations.

 Hence our topic in this paper will be the initial-value problem
 \be\la{5}
\begin{array}{rcl}
u_t+u\nabla u +\nabla p-\nu Y\! u & = & 0\,,\\
\div u & = & 0\,,\\
u|_{t=0} & = & u_0\,.
\end{array}
\ee

We will consider three cases:

\sm
\noindent
(i) {\it $K$ is finite}\\
 In this case we show that there exists a unique smooth solution $u(t)$ which converges to a steady-state solution of the Euler equation with all non-trivial Fourier modes supported in $K$, see Theorem~\ref{thm1}. An important part of the proof of this result is Theorem~\ref{fsupport} which says that any solution of the 2d incompressible Euler equation on $\tor$ which is supported on finitely many Fourier modes is a steady state.

 \sm
 \noindent
 (ii) {\it $K$ is cofinite} (i.\ e.\ ${\ZZ}^2\setminus K$ is finite)\\
 In this case we can establish an analogy of the Yudovich existence and uniqueness theorem for 2d incompressible Euler solutions when the initial vorticity $\om_0$ is bounded, i.\ e.\ $\om_0\in L^\infty(\tor)$, see Theorem~\ref{thm2}. However, the presence of the operator $Y$ in the equation leads to weaker estimates than those known for $Y=0$. In particular, we were unable to show that the norm $||\om(t)||_{L^\infty}$ will stay bounded. The dynamics for $t\to\infty$ is also unclear without extra assumptions.

 \sm
 \noindent
 (iii) {\it Both $K$ and ${\ZZ}^2\setminus K$ are infinite }\\
 Here can show the existence of weak solutions, but their uniqueness is unclear.

 \section{Preliminaries}\label{prelim}
 We keep the notation from the previous section and, in addition, we define an operator $Z$ as follows
 \be\la{6}
 \widehat{(Z u)}(k)=
 \begin{cases}
-|k|^2 \,\hat u(k)  \, & k\in K\,,\\
\hskip10pt{0}   & k\notin K\,,
\end{cases}
 \ee
 so that
 \be\la{6b}
 Y+Z=\dd\,.
 \ee
 In what follows we will assume that $\nu>0$ and we will also consider a small parameter $\ve>0$. The function spaces considered below will be spaces of functions (possibly vector-valued) on the torus $\tor$.
 \begin{proposition}\label{prop1}
 Let $K$ be any symmetric subset of ${\ZZ}^2$. For each div-free vector field $u_0\in L^2$ on the torus~$\tor$ the initial-value problem
 \be\la{7}
\begin{array}{rcl}
u_{\ve t}+\ue\nabla \ue +\nabla p_\ve-\nu Y\! \ue -\ve Z \ue & = & 0\,,\\
\div \ue & = & 0\,,\\
\ue|_{t=0} & = & u_0\,.
\end{array}
\ee
has a unique solution $\ue$ such that
\be\la{8}
\ue \in C([0,\infty),L^2_x)\cap L^2_t\dot H^2_x\,.
\ee
The solution $\ue$ is smooth in $\tor\times (0,\infty)$ and satisfies the energy identity
\be\la{9}
\inttor \frac12\,|\ue(x,t)|^2\,dx+\int_0^t\inttor \left(-\nu (Y\!\ue) \ue-\ve (Z\ue) \ue \right)\,dx\,dt'=\inttor\frac12 \,|u_0(x)|^2\,dx\,
\ee
for each $t\ge 0$.
Moreover, if $\om_0=\curl u_0\in L^2$, then $\ue\in L^2_t \dot H^2_x$ and satisfies
\be\la{10}
\inttor \frac12\,\ome^2(x,t)\,dx+\int_0^t\inttor \left(-\nu (Y\!\ome)\ome-\ve(Z\ome)\ome\right)\,dx\,dt'=\inttor\frac12\,\om_0^2(x)\,dx\,
\ee
for each $t\ge 0$, where $\ome=\curl \ue$.
\end{proposition}
 \begin{proof}
 This can be proved by classical energy estimates techniques for the 2d Navier-Stokes, see for example~\cite{Ladyzhenskaya}. The proof is slightly harder than for 2d Navier-Stokes without boundaries, as we do not have the maximum principle for the vorticity $\om=\curl u$, which we would have for the usual Navier-Stokes. Therefore we have to work purely with the energy estimates, similarly to the situation for 2d Navier-Stokes with boundaries (where use of the maximum principle is not very useful near boundaries, due to lack of control of the boundary condition for $\om$).
 In our case here we can proceed as follows. First, we show that the corresponding linear problem
 \be\la{lin}
 \begin{array}{rcl}
 u_{\ve t} +\nabla p_\ve-\nu Y\! \ue -\ve Z \ue & = & \div f\,,\\
\div \ue & = & 0\,,\\
\ue|_{t=0} & = & u_0\,.
 \end{array}
 \ee
 is uniquely solvable for $u_0\in L^2$ and $f\in L^2_tL^2_x$ in $C([0,\infty),L^2)\cap L^2_t\dot H^1_x$, with a bound on the norm given by $C(||u_0||_{L^2_x}+||f||_{L^2_tL^2_x})$.
 We now proceed along the usual lines and write~\rf{7} in the form~\rf{lin} with $f=-u\otimes u$ and invert the linear operator, to get an equation of the form
 \be\la{inteq}
 u=U+B(u,u)
 \ee
 where $U$ is the solution of~\rf{lin} for $f=0$ and $B(u,u)$ is the solution of~\rf{lin} with $u_0=0$ and $f=-u\otimes u$.
 We then recall the imbedding (see \cite{Ladyzhenskaya})
 \be\la{imb1}
 L_t^\infty L^2_x\cap L^2_t\dot H^1_x\subset L^4_tL^4_x\,,
 \ee
 and consider~\rf{inteq} as an equation in $L^4_tL^4_x$ in the space-time domain
 $\tor\times(0,T)$ for a sufficiently small $T$ such that $u\to U+B(u,u)$ is a contraction in $L^4(\tor\times(0,T))$. As the $L^2_x-$norm for the solution is non-increasing, this procedure gives a global solution by routine arguments.
 Obtaining smoothness of the solution is easier than the corresponding problem for 2d Navier-Stokes in domains with boundaries, as we can take derivatives of our equation in any direction. (Note that the operators $Y$ and $Z$ commute with derivatives.)

 \end{proof}

\sm
\noindent
By a {\it weak solution} of the initial value problem~\rf{5} with $u_0\in H^1$ we mean a function
\be\la{11}
u\in C([0,\infty),L^2)\cap L^\infty_t\dot H^1_x\cap\{v, \, Y\!v\in L^2_tL^2_x\}\,
\ee
which satisfies the equation in the sense of distribution for a suitable $p$, and attains the initial datum $u_0$ at $t=0$. (Note that the last condition is well-defined, due to the continuity of $u(t)$ as an $L^2-$valued function.)

\begin{corollary}
For any symmetric set $K\subset{\ZZ}^2$ and any initial datum $u_0\in H^1$ the initial-value problem~\rf{5} has at least one weak solution.
\end{corollary}
\begin{proof}
This follows by letting $\ve\searrow0$ and choosing a subsequence $\ue$ which converges strongly in $L^2_tL^2_x$. The pre-compactness in $L^2_tL^2_x$ of the family $\ue$ for $0<\ve<\ve_0$ follows from~\rf{10}, standard imbeddings and the Aubin-Lions lemma (the usual tool for showing that the control of the time derivative in weak spaces, together with a suitable compact imbedding in $x$ is enough for pre-compactness, see~\cite{Aubin, Lions}).
\end{proof}

The uniqueness of the solutions seems to be open except in the case when $K$ is finite, which will be addressed in the next section. The case when $K$ is co-finite and, in addition, $\om_0\in L^\infty$ can also be handled, due its similarities with Euler's equation and the Yudovich theorem~\cite{Yudovich}, see Section~\ref{cofinite}.

\section{Finitely many undamped frequencies}\label{finite}
Let us now assume that the set $K$ of the undamped frequencies is finite. In this case one expects that the high spatial frequency behavior of the solution is the same as for Navier-Stokes and if the forcing term vanishes, as we assume here, the norm
$||u(t)||_{L^2_x}$ will be non-increasing, due to the energy estimate. Let
\be\la{kappa}
\kappa=\max_{k\in K} |k|\,.
\ee
Clearly
\be\la{3-1}
||\nabla v||^2\le -(Y\! v, v) + \kappa^2 ||v||_{L^2_x}^2
\ee
for each $v\in H^1(\tor)$. Hence estimate~\rf{9} in Proposition~\ref{prop1} given uniform control of the solutions $\ue$ in $L^\infty_tL^2_x\cap L^2_t\dot H^1_x$, just as in the case of standard Navier-Stokes. A minor modification of standard 2d Navier-Stokes arguments now gives existence and uniqueness of solutions to the initial-value problem~\rf{5} with $u_0\in L^2$ in the class $C([0,\infty), L^2)\cap L^2_t\dot H^1_x$, together with their smoothness for $t>0$. An interesting question about these solutions is their long-time behavior. We address this in the next theorem, which also recapitulates the existence result just discussed.

We introduce the following notation. In the following definition a solution of the Euler equation means a vector field satisfying the equation for a suitable pressure function. For~$E,I\ge0$ we let
\be\la{ekeidef}
\EE_{K,E,I}=\left\{
v\colon\!\tor\!\to\!\R^2,\begin{array}{l}
\hbox{$v$ solves 2d steady incompressible Euler and, in addition,}\\

\hbox{$\hat v(k)=0$ for each $k\notin K$} \,,\,
\inttor |v|^2=2E\,,\,
\inttor|\curl v|^2=I.
\end{array}\!\!\right\}
\ee

 As $K$ is finite, $\EE_{K,E,I}$ is defined by some polynomial equations on a finite-dimensional space, and therefore it is a finite-dimensional algebraic variety (possibly with some non-smooth points).

We will say that the set $K$ is {\it degenerate} if all its points lie on a circle centered at the origin, or all its points lie on a line passing through the origin. The condition is equivalent to saying that any div-free velocity field with Fourier support in $K$ is a steady solution of Euler's equation, see Section~\ref{modes} for a more detailed discussion.
\begin{theorem}\label{thm1}
Assume the set $K$ of the undamped frequencies is symmetric and finite. Then for each div-free vector field on the torus $u_0\in L^2$  the initial-value problem~\rf{5} has a unique solution $u$ in the class $C([0,\infty), L^2)\cap L^2_t \dot H^1_x$ which satisfies the energy identity
\be\la{3-2}
\inttor \frac 12\,|u(x,t)|^2\,dx+\int_0^t\inttor -\nu (Y\!u\,,\,u)\,dx\,dt'=\inttor \frac 12 |u_0(x)|^2\,dx\,,
\ee
for each $t\ge 0$. The solution is smooth for $t>0$. As $t\to \infty$ the trajectory $u(t)$ stays in a compact subset of $C^k$ for each $k=1,2,\dots$ and its $\omega-$limit set is a subset of a connected component of $\EE_{K,E,I}$, where
\be\la{3-6}
E=\lim_{t\to\infty} \inttor \frac 12 \,|u(x,t)|^2\,dx\,,\qquad I=\lim_{t\to\infty}\inttor \om^2(x,t)\,dx\,.
\ee
 In addition, if the set $K$ is degenerate and $I$ is sufficiently small (depending on $K$ and $\nu$), the $\om$-limit set of the solution $u(t)$ consist of a single point which is approached exponentially in any $C^k$-norm.

\end{theorem}
\begin{proof}
The issues concerning the existence and regularity having been addressed above, we turn to the statement concerning the long-time behavior. The compactness of the trajectory $u(t)$ in $C^k$  follows from the $L^2$-estimate for $\om=\curl u$ (enstrophy estimate)
\be\la{3-7}
\inttor \frac12\,\om^2(x,t)\,dx+\int_0^t\inttor -\nu\left( Y\!\om,\om\right)\,dx\,dt'=\inttor\frac12\,\om_0^2(x)\,dx\,\,,
\ee
together with the regularity theory for 2d Navier-Stokes on finite time intervals (with minor modifications accounting for our situation here).
Let $z$ be a field in the $\omega-$limit set of the trajectory $u(t)$. We have $u(t_j)\to z$ in $L^2$ for some sequence $t_j\to \infty$. Let us define $u_j(t)$ by
\be\la{3-8}
u_j(t)=u(t+t_j)\,.
\ee
It is easy to see from regularity estimates for our solutions that for large $j$ the solutions $u_j$ are bounded for $t\ge 0$ together with all their derivatives and hence converge to a solution $\bar u$ such that $\bar u|_{t=0}=z$. As both the energy and the enstrophy are Lyapunov functions for the evolution and are continuous with respect to our convergence, we see that on the solution $\bar u$  the enstrophy and energy are constant, equal to $E$ and $I$ respectively. Hence the integral $\inttor (Y\!\bar u(t),\bar u(t))\,dx$ vanishes for each $t\ge 0$.  This  shows that the Fourier modes of $\bar u$ are supported in $K$, and $\bar u$ solves the incompressible Euler equation (as $Y$ vanishes on functions with Fourier support in $K$).  By virtue of Theorem~\ref{fsupport} we know that $\bar u$ has to be a steady state, and therefore $\bar u(t)\equiv z$.

Let us now assume that $K$ is degenerate. We will write our equation for $\om(t)$ as
\be\la{d1}
\om_t=b(\om,\om)+\nu Y\!\om(t),
\ee
where
\be\la{d2}
b(\om^{(1)},\om^{(2)})=-u^{(1)}\nabla \om^{(2)}\,,
\ee
with $u^{(1)}$ being the velocity field generated by $\om^{(1)}$ via the usual Biot-Savart law. Let us write
\be\la{d3}
\om(t)=\om_{D}(t)+\om_{K}(t)\,,
\ee
with
\be\la{d4}
\om_K(t)=\frac1{2\pi} \sum_{k\in K} \hat\om (k,t) e^{ikx}\,.
\ee
As the set $K$ is degenerate, we have
\be\la{d5}
b(\om_K(t),\om_K(t))=0\,.
\ee
Therefore
\be\la{d6}
\om_t=b(\om_D,\om_D)+b(\om_D,\omk)+b(\omk,\om_D)+\nu Y\!\om(t)\,.
\ee
Denoting by $(f,g)$ the $L^2-$ scalar product $\inttor fg\,dx$, we note that at each time we have
\be\la{d7}
(b(\om_D,\om_D),\om_D)=(b(\om_K,\om_D),\om_D)=0\,.
\ee
Therefore
\be\la{d8}
\begin{split}
\frac d {dt} ||\om_D||^2_{L^2} & =(b(\om_D,\om_K),\om_D)+\nu (Y\!\om_D,\om_D)\\ & \le c_K||\om_K||_{L^2} ||\om_D||_{L^2}^2+\nu(Y\!\om_D,\om_D)\\
&
\le (c_K||\om||_{L^2}-\nu\kappa)\,||\om_D||_{L^2}^2\,,
\end{split}
\ee
where all quantities are evaluated at time $t$,
\be\la{d9}
\kappa=\min_{k\in\zt\setminus K}|k|^2\,,
\ee
and $c_K$ is the best constant $c$ in the inequality
\be\la{d10}
||u_D\nabla \om_K||_{L^2}\le c||\om_D||_{L^2}||\om_K||_{L^2}\,,
\ee
which is obviously finite, as due to the finiteness of $K$ we have
\be\la{d11}
||\nabla \om_K||_{L^\infty}\le c'_K ||\nabla \om_K||_{L^2}\,
\ee
for some (finite) constant $c'_K$.
We see from~\rf{d8} that if
\be\la{d12}
I<\frac{\nu\kappa}{c_K}\,,
\ee
the function $\om_D(t)$ approaches zero exponentially in $L^2$. At the same time, given any $k=0,1,2,\dots$, the derivatives $\nabla^k\om(t)$ are bounded point-wise by regularity  and the derivatives $\nabla^k\om_K(t)$ are also bounded point-wise due to the finiteness of $K$. This means that the derivatives $\nabla^k\om_D(t)$ are also bounded point-wise, and from the exponential $L^2-$convergence of $\om_D(t)$ to zero, we can therefore infer point-wise exponential convergence of $\om_D(t)$ and all its derivatives to zero.
Equation~\rf{d6} then shows point-wise exponential convergence of $\om_t$ (and all its derivatives) to zero, and the proof is finished.
\end{proof}
\begin{remark}
It seems to be likely that the $\omega-$limit set of each trajectory always consists of a single point, but a possible proof of this seems to require a more subtle analysis. Note that for $E,I>0$ the set $\EE_{K,E,I}$, when non-empty, will be at least one-dimensional, due to translational invariance.
\end{remark}

\begin{remark}
If $K$ is degenerate and the equilibria of Euler equation with the Fourier support in $K$ are unstable for Euler evolution (which is expected to be the case unless $K$ consists of the lowest non-trivial modes), it is very likely that for generic solutions the quantity $I$ will indeed be small, as required by the last statement of the theorem. This is because until the viscosity  stabilizes the behavior as in the proof above, the trajectory approaching the equilibria can be expected to be repeatedly repelled into the more dissipative regime. A rigorous proof again seems to require a more subtle analysis of the dynamics. If the equilibria in $K$ are stable, then the quantity $I$  might not typically approach small values, although it could again be non-trivial to prove this.

In any case, it seems to be reasonable to conjecture, that for degenerate sets $K$ the solution will not generically approach zero.
\end{remark}

The Kraichnan picture of 2d turbulence suggests that all the Euler equilibria other then the ones given by the lowest non-trivial modes should be unstable, and the lowest modes are stable only when the energy is fixed. It is therefore  conceivable that a generic solution of~\rf{5} for finitely many undamped modes will first approach the lowest modes as a kind of meta-stable state, and -- assuming these are not damped -- then very slowly approach zero. It would of course be of interest to understand the linearized stability of the equilibria and their stable, unstable, and center manifolds.

\section{Finitely many damped frequencies}\label{cofinite}
In this section we assume that $K$ is co-finite, i.\ e.\ only finitely many frequencies are damped. The equation should then be close to Euler equation.
Our main result in this case is the following.
\begin{theorem}\label{thm2}
Assume the set $K$ is symmetric and co-finite and the initial vorticity $\om_0=\curl u_0$ belongs to $L^\infty$. Then the system~\rf{5} has a unique global solution $u$ such that $\om=\curl u$ belongs to $L^\infty_tL^\infty_x$ on any finite time interval. The solution  $u$ satisfies the energy estimate~\rf{9} with $\ve=0$, the enstrophy estimate~\rf{10} with $\ve=0$ and, for some $c_K\ge 0$ which depends only on $K$, we have
\be\la{4-1}
||\om(t)||_{L^\infty}\le ||\om_0||_{L^\infty} + c_K||\om_0||_{L^2}\,\,t\,,\qquad t>0\,.
\ee
The trajectory $u(t)$ is pre-compact in $L^2$ and if $u(t_j)\to z$ in $L^2$ for some sequence $t_j\to\infty$ and $\curl z\in L^\infty$, then the solutions $u_j(t)=u(t+t_j)$ converge as $j\to\infty$ in $L^\infty_{t}L^2_x$ on any finite time interval to a solution of the (incompressible) Euler equation whose Fourier coefficients are supported in $K$.
\end{theorem}

\begin{proof}
The energy and the enstrophy bounds are obvious. To prove~\rf{4-1}, we will write our equation in the vorticity form as
\be\la{4-2}
\om_t+u\nabla\om=\nu Y\!\om
\ee
and use the bound
\be\la{4-3}
||Y\!\om||_{L^\infty}\le c_K ||\om||_{L^2}\,,
\ee
which is an immediate consequence of the finiteness of $K$. As $||\om(t)||_{L^2}$ is non-increasing, \rf{4-2} then gives
\be\la{4-4}
\frac d{dt} ||\om(t)||_{L^\infty}\le c_K||\om_0||_{L^2}
\ee
and our bound follows.  Strictly speaking, one needs some regularity of the solution for this proof to work. One way to deal with this issue (in case we do not wish to get into details of justifying the above reasoning under minimal regularity assumptions) is to prove the corresponding bound for the solutions of the regularized equation~\rf{7}. We note that the first equation of~\rf{7} in the vorticity form can be written as
\be\la{4-5}
\om_{\ve t}+\ue\nabla\om_\ve-\ve\dd\om_\ve=(\nu-\ve)Y\!\om_\ve
\ee
and, as $\om_\ve$ is smooth for $t>0$, we have (similarly as above)
\be\la{4-6}
\frac d{dt} ||\om_\ve(t)||_{L^\infty}\le c_K||\om_0||_{L^2}\,.
\ee
This clearly implies the bound~\rf{4-1} for the limiting solution $\om$ as $\ve\searrow 0$. We will next prove uniqueness, and therefore it is enough to have the bound for the particular solution obtained as a limit of (a subsequence of) the solutions $\om_\ve$.

The uniqueness follows from the argument below for the last statement of the theorem. It gives a somewhat stronger statement than uniqueness in that it shows a stability property of the solutions with respect to perturbations of initial conditions (in some specific classes of functions). The main idea goes back to Yudovich~\cite{Yudovich}, see also~\cite{MajdaBertozzi}.

Let $\bar u$ be a solution of~\rf{5} with $\bar u(0)=z$ such that its vorticity $\bar\om$ satisfies the bound~\rf{4-1} with $\om_0$ replaced by $\curl z$. At this point it may not be clear that the solution is unique, but we know at least one such solution exists. We will show that $u_j$ approach $\bar u$ on any finite time interval. It will be clear that the argument also gives uniqueness of $\bar u$.
We let
\be\la{4-7}
v_j=u_j-\bar u\,.
\ee
We have , for a suitable function $q_j=q_j(x,t)$,
\be\la{4-8}
v_{jt}+u_j\nabla v_j+v_j\nabla \bar u +\nabla q_j-\nu Y\!v_j=0\,.
\ee
Taking scalar product with $v_j$ and integrating over $\tor$ we obtain
\be\la{4-8b}
\frac {d}{dt}\inttor |v_j(x,t)|^2\,dx \le 2\inttor |\nabla\bar u(x,t)|\,\,|v_j(x,t)|^2\,dx\,,
\ee
where we used that $\inttor (Y\! v_j, v_j)\,dx\le 0$. We now write
\be\la{4-9}
\inttor|\nabla \bar u(t)|\,\,|v_j(t)|^2\,dx\le ||\nabla \bar u(t)||_{L^p}||v_j(t)||_{L^{2p'}}^2\,,
\ee
where $p>1$ and , as usual, $1/p+1/p'=1$. For any $r>2$ and $p'<r$ we have by elementary interpolation (or H\"older inequality)
\be\la{4-10}
||v_j(t)||_{L^{2p'}}\le ||v_j(t)||_{L^2}^{1-\frac{r}{(r-2)p}}||v_j(t)||_{L^r}^{\frac{r}{(r-2)p}}\,.
\ee
Let us fix $T>0$ and let us assume that
\be\la{4-11}
||\bar \om(t)||_{L^\infty}\le M_1\,,\qquad t\in (0,T)\,.
\ee
Yudovich's trick now is to use  the estimate
\be\la{4-12}
||\nabla \bar u(t)||_{L^p}\le p A||\bar \om(t)||_{L^p}\,,
\ee
where $A>0$ is a suitable constant independent of $p$.
This gives
\be\la{4-13}
||\nabla \bar u(t)||_{L^p}\le pM_2\,,\qquad t\in(0,T)\,,
\ee
for a suitable constant $M_2>0$ independent of $p$.
We also note that for any fixed $r>2$ the norms $||v_j(t)||_{L^r}$ are uniformly bounded in $j$, thanks to the enstrophy estimates for $u_j$.
Letting
\be\la{4-14}
y_j=y_j(t)=||v_j(t)||^2_{L^2}\,,\,\qquad \ve = \frac {r}{(r-2)p}\,,
\ee
we thus conclude that for some $M>0$ we have
\be\la{4-15}
\ve \frac{d y_j}{dt}  \le My_j^{1-\ve}
\ee
This means that
\be\la{4-16}
y_j(t)\le \left[(y_j(0))^\ve+Mt\right]^{\frac 1\ve}\,.
\ee
As $y_j(0)\to 0$ for $j\to\infty$, it is easy to see that $y_j(t)$ converges to zero locally uniformly for $0\le t< 1/M $. The argument can now be repeated, and after at most $MT+1$ iterations we get uniform convergence of $v_j$ to zero in $L^2$ on the interval $[0,T]$.

Letting
\be\la{4-18}
E=\lim_{t\to\infty} ||u(t)||^2_{L^2}
\ee
(which exists, as $||u(t)||_{L^2}$ is non-increasing), we clearly have
\be\la{4-19}
||z||_{L^2}^2=E\,,\qquad ||\bar u(t)||_{L^2}=E\,,\qquad t>0\,.
\ee
Therefore
\be\la{4-20}
(Y\!\bar u(t)\,,\,\bar u(t))=0\,,\qquad t>0\,,
\ee
and we see that the Fourier coefficients of $\bar u(t)$ must be supported outside of $K$ for each $t>0$.
\end{proof}
\begin{remark}
All what is needed in the uniqueness/stability proof is that one can derive~\rf{4-8b} (for which a certain amount of regularity is needed), that $||\bar \om(t)||_{L^\infty}$ is bounded in $(0,T)$, and that $||u_j(t)||_{L^r}$ are uniformly bounded for some $r>2$. Hence the uniqueness/stability proof  gives a {\it weak-strong} result, in the sense that the regularity assumptions are mostly needed just for one solution. As the examples of~\cite{CamilloLaszlo} show, one cannot get a weak-strong uniqueness result if the weaker solution is  merely  assumed to be H\"older continuous. This is related to Onsager's conjecture (as discussed in~\cite{CamilloLaszlo}).
\end{remark}

\section{Euler solutions  supported on finitely many Fourier modes}\label{modes}
\begin{theorem}\label{fsupport}
If $u(x,t)$ is a solution of 2d incompressible Euler's equation on $\tor$ which is supported on finitely many Fourier modes, then $u$ is independent of time. Moreover, its ``Fourier support" is either a subset of a circle centered at the origin, or a line passing through the origin.
\end{theorem}

\begin{proof}
The Fourier coefficients of any real-valued integrable functions (or, in fact, any distribution) $f\colon\tor\to\R$ have the property that \be\la{5-1}
\hat f(-k)=\overline{\hat f(k)}\,,
\ee
which we will use in the proof.  As in this paper we are dealing with real-valued solutions, we will take advantage of the simplifications offered by~\rf{5-1}, although the statement remains true for complex solutions.
We will use the following terminology.
\begin{itemize}
\item[(i)] The {\it Fourier support} of a function $f\colon\tor\to \R$ is the set $\{k\in {\ZZ}^2\,,\,\hat f(k)\ne 0\}$.
\item[(ii)] A set $S\subset {\ZZ}^2 $ is {\it symmetric} if it is invariant under the map $k\to -k$.
\item[(iii)] An (unordered) pair $\{k,l\}$ of two distinct points $k=(k_1,k_2),l=(l_1,l_2)\in \zt\setminus\{(0,0)\}$ is called {\it degenerate} if either $k,l$ lie on the same circle centered at the origin \hbox{(i.\ e.\ $k_1^2+k_2^2=l_1^2+l_2^2$),} or $k,l$ lie on the same line passing through the origin (i.\ e.\  $k_1l_2-k_2l_2=0$). In the former case we call the pair to be {\it c-degenerate}, and in the latter case we call the pair to be {\it l-degenerate}. A pair which is not degenerate is called {\it non-degenerate}.

    In what follows we will only be dealing with unordered pairs, and therefore by ``pair" we will always mean an unordered pair.
\item[(iv)] A {\it degenerate set} $S\subset \zt$ is a set for which each two-point subset $\{k,l\}\subset S$ is degenerate.

\end{itemize}

We will be dealing with Fourier support of vorticity fields and it is useful to keep in mind that {\sl the Fourier support of a vorticity field never contains the origin $(0,0)$}, as $\inttor \curl u = 0$.

It is easy to see that {\sl a finite subset of $\zt$ which does not contain the origin is degenerate if and only if it is a subset of a circle centered at the origin or a line passing thought the origin.} In other words, if all  pairs in a set $S$ are degenerate, then all of them are c-degenerate or all of them are l-degenerate. (We are not claiming that these two possibilities are mutually exclusive.)

From~\rf{5-1} we see that the Fourier support of a real function on $\tor$ is always a symmetric subset of $\zt$.

A particularly useful way to write the Euler equation for a vorticity field $\om(x,t)$ on $\tor$ in our context here is the following, see for example~\cite{HM}:
\be\la{5-2}
\frac d{dt} \,\hat \om(m,t)=-\frac 1{4\pi}\sum_{k+l=m}\left(k_1l_2-k_2l_1\right)\left(\frac1{|k|^2}-\frac1{|l|^2}\right)\hat\om(k,t)\hat\om(l,t)\,.
\ee
Let us now consider the vorticity field $\om(x,t)$  of a solution of the Euler equation on some time interval $(t_1,t_2)$ satisfying our assumptions and let $S\subset\zt$ be a finite set such that the Fourier support of $\om(t)$ is contained in $S$ for each $t\in(t_1,t_2)$. We can assume that $S$ is the minimal set with this property, which means that for each $k\in S$ we have $\hat\om(k,t)\ne 0$ for some $t\in(t_1,t_2)$. Due to the finiteness of $S$, we are dealing with a solution of an finite-dimensional ODE given by polynomials, and hence all functions $t\to \hat\om(k,t)$ are analytic. Together with the minimality of $S$ this implies that the set $Z\subset\tjd$ of times where at least one of the functions $\hat \om(k,t)\,,\,k\in S,$ vanishes cannot have an accumulation point in $\tjd$.

To show that the set $S$ must be degenerate, it is enough to establish the following statement:

\sm
\noindent
\parbox{1pt}
 \ ($*$)\ \parbox{3pt}\ \ \parbox{5in}{\sl If the set $S$ is not degenerate, then there exists a non-zero element $m\in\zt\setminus S$ such that $m=k+l$ for exactly one  non-degenerate pair $\{k,l\}\in S$.}

 Once ($*$) is proved, then for a set $S$ which is not degenerate and $m$ as in ($*$) expression~\rf{5-2} gives
 \be\la{5-4}
 \frac d{dt}\hat \om(m,t)\ne 0\,\qquad t\in \tjd\setminus Z\,,
 \ee
 contradicting  $m\notin S$. Hence $S$ must be degenerate.

 It remains to prove ($*$). Arguing by contradiction, let us assume that we have a finite set $S\subset\zt\setminus\{(0,0)\}$ as above which is not degenerate such that either $S$ is closed under addition of its non-degenerate pairs, or each non-zero $m\in \zt\setminus S$ with $m=k+l$ for some non-degenerate  pair of points $\{k,l\}$ in  $S$ can also be expressed as $m=k'+l'$ for a different non-degenerate  pair of points $\{k',l'\}$ in $S$.

  We can consider $S$ as a subset of $\R^2$ (into which $\zt$ is naturally imbedded) and let $\sconv$ be the convex hull of $S$. As $S$ is symmetric, $\sconv$ is also symmetric (i.\ e.\ invariant under $k\to-k$). If it consist of a line, then $S$ is obviously degenerate, and therefore the only non-trivial case is when $\sconv$ is a symmetric convex polygon with the origin belonging to its interior.  By classical results about convex sets, there is a unique norm $N$ on $\R^2$ such that
  \be\la{5-5}
  \sconv=\{x\in \R^2\,, \, N(x)\le 1\}.
  \ee
  Let $A_1,\dots, A_r$ be the extremal points of $\sconv$.

   The proof will  proceed in two steps.

   \sm
   {\noindent}
   {\it Step 1}

   \sm
   \noindent
   We show that {\sl  the points $A_1,\dots, A_r$ lie on a circle centered at the origin and, moreover,}
  \be\la{5-6}
  \{x\in\R^2\,,\,N(x)=1\}\cap S = \{A_1,\dots, A_r\}\,.
  \ee
 (Note that~\rf{5-6} amounts to saying that  the segments at the boundary of the polygon $\sconv$ do not contain any points of $S$ other than the extremal points.)

  To prove this, let us assume without loss of generality that the boundary of our polygon $\sconv$ is given by the segments $[A_1,A_2], [A_2,A_3],\dots,[A_{r-1},A_r], [A_r,A_1]$. Denoting by $|A|$ the euclidean distance of $A$ from the origin, we aim to show that
  \be\la{5-7}
  |A_1|=|A_2|\,\qquad \hbox{and $\qquad [A_1,A_2]\cap S=\{A_1,A_2\}$\,.}
  \ee
  By elementary convex analysis we  can find a linear function $L$ on $\R^2$ such that
  \be\la{5-7b}
  L(A_1)>L(A_2)>\max_{P\in (S\setminus[A_1,A_2])\cup \{O\}} L(P)\,\,,
  \ee
  where $O$ denotes the origin. Let $A_1=P_1, P_2,\dots, P_s=A_2$ be the list of all points in $S\cap [A_1,A_2]$ ordered so that
  \be\la{5-8}
  L(P_1)>L(P_2)>\dots>L(P_s)\,.
  \ee
  As the line passing through the points $A_1,A_2$ cannot pass through the origin, the set $\PP=\{P_1,\dots,P_s\}$ contains no l-degenerate pairs. If all pairs $\{P_1,P_j\}\,,\,j\ge 2$ are degenerate, they must always be c-degenerate, and hence they all lie on one circle centered at the origin. This is only possible if $\PP=\{A_1,A_2\}$ and $|A_1|=|A_2|$. If not all pairs $\{P_1,P_j\}$ are degenerate, let us consider the smallest index $1<j\le r$ such that $\{P_1,P_j\}$ is not degenerate and the point $Q=P_1+P_j$. Clearly
  \be\la{5-9}
  L(Q)=L(P_1)+L(P_j)\ge L(A_1)+L(A_2)>L(A_1)
  \ee
  and hence $Q\notin S$. Moreover, if
  \be\la{5-10}
  Q=P'+P''
  \ee
  for some  non-degenerate pair $\{P',P''\}$ of points of $S$  we must have
  \be\la{5-11}
  L(P')+L(P'')=L(P_1)+L(P_j)\,.
  \ee
  In view of~\rf{5-7b} and~\rf{5-8} this is possible only if $\{P',P''\}\subset\{P_1,P_2,\dots, P_j\}$. As the  pairs $\{P_1,P_2\},\{P_1,P_3\},\dots\{P_1,P_{j-1}\}$ are all c-degenerate by the definition of $j$, the points $P_1,\dots,P_{j-1}$ must all lie on one circle centered at the origin (and we also see that $j\le 3$, which however will not be needed). Therefore, taking into account~\rf{5-8}, the only possibility for the  pair $\{P',P''\}$ to be both non-degenerate and satisfy~\rf{5-11} is $\{P',P''\}=\{P_1,P_j\}$. We see that the point $Q$ does not belong to $S$ and can be expressed as a sum of a non-degenerate pair in $S$ in exactly one way.
  This contradiction concludes Step 1 of our proof.

  \sm
  \noindent
  {\it Step 2}

  \sm
  \noindent
  We show that
  \be\la{5-12}
  S=\{A_1,\dots, A_r\}\,.
  \ee
Arguing again by contradiction, assume that this is not the case and consider $B\in S$ such that
\be\la{5-13}
0<N(B)=\max_{P\in S\setminus\{A_1,\dots, A_r\}} N(P)\,.
\ee
We can assume without loss of generality that $B$ is in the (closed) triangle $OA_1A_2$, where $O$ is again the origin. In fact, denoting by $C$ the center of the segment $[A_1,A_2]$, we can assume that $B$ belongs to the (closed) triangle $OCA_1$. Let $M$ be the linear function on $\R^2$ such that $M=N$ on the triangle $OA_1A_2$. Consider now the  pair $\{B,A_2\}$ and the point $Q=A_2+B$. The pair $\{B, A_2\}$ is clearly non-degenerate and
\be\la{5-14}
1=M(A_1)=M(A_2)>M(B)\ge\max_{P\in (S\setminus\{A_1,\dots,A_r\})\cup \{O\}}M(P)\,.
\ee
Assume
\be\la{5-14b}
Q=A_2+B=P'+P''
\ee for another non-degenerate pair $\{P'\,,\,P''\}$ in $S$.
Then
\be\la{5-15}
M(P')+M(P'')=M(B)+M(A_2)\,,
\ee
and as one of the points $P',P''$ must belong to $S\setminus\{A_1,\dots,A_r\}$ (the set $\{A_1,\dots,A_r\}$ being degenerate by Step 1), we see from~\rf{5-13}, the equality $M(B)=N(B)$, \rf{5-14}, and~\rf{5-15} that after switching the r\^ole of $P',P''$, if necessary, we may assume that
\be\la{5-16}
P'\in\{A_1,A_2\}\qquad\hbox{ and $\qquad M(P'')=M(B)$\,.}
\ee
The case $P'=A_2$ can be ruled out, since~\rf{5-14} would then imply that $P''=B$, and we would have $\{P',P''\}=\{A_2,B\}$.
This means that $P'=A_1$ and, by~\rf{5-14b},
\be\la{5-17}
P''=B+(A_2-A_1)\,.
\ee
We see that $P''$ is obtained from $B$ by a shift along the direction $A_2-A_1$ by the length of the segment $[A_1,A_2]$. However, the side of the convex polygon given by
$\{x\,,\,N(x)\le N(B)\}$ contained in the triangle $OA_1A_2$ is strictly shorter that $|A_2-A_1|$. Therefore $N(P'')>N(B)$. By definition of $B$ this means that $P''\in\{A_1,\dots, A_r\}$ and by Step 1 the pair $\{P',P''\}$ is degenerate. This is again a contradiction and the proof is finished.
\end{proof}

\begin{remark}
With some additional reasoning, our argument above really shows that if $\om(t)$ is a solution of Euler equation which is in some uniqueness class and its Fourier coefficients are supported in a fixed bounded set $S\subset\zt$ for a set of times which has a finite accumulation point, then $\om(t)$ has to be constant in time. Does there exist a non-stationary solution of 2d Euler equation in $\tor$ such that the Fourier support of $\om(t)$ is bounded for two different times $t_1$ and $t_2$? We do not know the answer to this question. One can also consider its variants for finitely many times $t_1<t_2<\dots <t_n$, or infinitely many times without a finite accumulation point.
\end{remark}

\begin{remark}
In addition to the stationary solutions with finite Fourier support, there are of course also many stationary solutions for which the Fourier support is infinite. This can be seen for example by considering solutions with $\om=\Delta \psi$, where $\psi$ satisfies
\be\la{5-18}
\Delta \psi = f(\psi)\,,
\ee
for suitable non-linear polynomial functions $f$.
\end{remark}

\bb

{\centerline {\bf Acknowledgements}}

\bb

The research of W.~H. was supported in part by grant DMS 1159376 from the National Science Foundation. The research of V.~S. was supported in part by grants DMS 1159376 and  DMS 1362467 from the National Science Foundation.

\bb

\bb
\noindent
\textsc{Department of Mathematics, Princeton University}

\smallskip
\noindent
\textsc{School of Mathematics, University of Minnesota}
\end{document}